\documentclass[a4paper, 12pt,oneside]{article}

\usepackage{bbm}
\usepackage{nicefrac}
\usepackage{hyperref,amsthm,amssymb,amsmath}

\newcommand*{\mailto}[1]{\href{mailto:#1}{\nolinkurl{#1}}}

\usepackage{graphicx}
\usepackage{wrapfig}
\usepackage{placeins} 
\usepackage{float}
\newtheorem{theorem}{Theorem}[section]

\newtheorem{lemma}[theorem]{Lemma}

\newcommand{\be}{\begin{equation}}
\newcommand{\ee}{\end{equation}}

\numberwithin{equation}{section}

\begin{document}
\begin{center}
{\Large \bf On recovering Dirac operators with two delays }
\end{center}
\begin{center}
{\bf Biljana Vojvodi\'{c},  Neboj\v{s}a Djuri\'{c} and Vladimir Vladi\v{c}i\'{c} }
\end{center}

{\bf Abstract:}
We study the inverse spectral problems of recovering  Dirac-type functional-differential operator with two constant delays $a_1$ and $a_2$ not less than one-third of the interval. It has been proved that the operator can be recovered uniquely  from four spectra  under the condition $2a_1+\frac{a_2}{2}\geq \pi$, while it is not possible otherwise.

{\bf Subclass:} 34A55, 34K29. 
\section{Introduction}
\label{sec:int}

The theory of differential equations with delays is a a very significant area of the theory of ordinary differential equations (see \cite{Mysh, Nor}). In last decades, there has been a growing interest in studying inverse spectral problems for different types of operators with one or more delays. It turned out that this type of operators is usually more adequate for modeling different real physical processes, frequently possessing a nonlocal nature. Inverse problems for Sturm-Liouville operators with one  delay have been studied in most details  (see \cite{BondYur18, ButYur19, DB21, DB21-2, DB22, DV19, FrYur12, Pik91, PVV19, VBV22, VlaPik}). There is also considerable number of results related to the Sturm-Liouville operators with two constant delays (see \cite{butc,ppv, vpm, vojj, vpv, vpvc, vv}). In recent years, a significant number of results related to the inverse problems for Sturm-Liouville operators, have been extended to Dirac operators  with one delay (see \cite{but, dju, wang1, wang2}), as well as to Dirac operators with two delays (see \cite{vvdju}). 
\\\\
The key issue in solving inverse problems for operators with delays is the question of inverse problem solution's uniqueness.  Although inverse problem solution's uniqueness was for long thought to be indisputable, as in the case of inverse problems for classical operators (without delays), it turned out that the solution of inverse problems for operators with delays does not have to be unique. It has been shown in the papers \cite{PVV19} and \cite{VlaPik} that Sturm-Liouville operator with one delay can be recovered uniquely from two spectra if the delay belongs to $[\frac{2\pi}{5},\pi),$  while in the papers \cite{DB21} and \cite{DB21-2} has been shown that this is not possible for the delay from $[\frac{\pi}{3},\frac{2\pi}{5}).$ There are the same results for Dirac operator with one delay (see \cite{but} and \cite{dju}). For operators with two delays, there are just results related to the inverse problem solution's uniqueness. So, in the papers \cite{vpv} and \cite{vpvc} it has been proven that  Sturm-Liouville operator can be recovered uniquely from four spectra for the delays greater than $\frac{\pi}{2}$ under Robin and Dirichlet boundary conditions, respectively. In the paper \cite{vojj} inverse problem solution's uniqueness has been proven for the Sturm-Liouville operator with two delays from $[\frac{2\pi}{5},\frac{\pi}{2})$ under Robin boundary conditions. Papers \cite{vpm} and \cite{vv} deal with Sturm-Liouville operator with two delays such that first delay $a_1$ belongs to $[\frac{\pi}{3},\frac{2\pi}{5})$ and the second one to  $[2a_1,\pi) $ under Robin and Dirichlet/Neumann boundary conditions, respectively and uniqueness of inverse problem solution has been proven. In the paper \cite{vvdju} it has been proven that Dirac operator  with two delays from $[\frac{2\pi}{5},\pi)$ can be recovered uniquely from four spectra.  So far  there are no results with non-unique solutions of inverse problems for operators with two delays, even for  Sturm-Liouville operators.This paper will be the first result proving that the uniqueness of the inverse problem's solution does not have to be valid neither for operators with two delays. 
\\\\
In this paper we study Boundary value problems (BVPs) $D_{j}(P,Q,m), m\in\{0,1\}, j\in\{1,2\},$ for Dirac-type system of the form
\begin{align}\label{a01}
    BY'(x)+(-1)^{m}P(x)Y(x-a_{1})+Q(x)Y(x-a_{2})=\lambda Y(x), x\in(0,\pi)
 \end{align}
\begin{equation*}\
y_{1}(0)=y_{j}(\pi)= 0
\end{equation*}
where
\[B=\left[
\begin{array}{cccc}
0 & 1\\
-1 & 0
\end{array}
\right],\hspace{2mm}
Y(x)=\left[
\begin{array}{cccc}
y_{1}(x) \\
y_{2}(x)
\end{array}
\right],\hspace{2mm}
\frac{\pi}{3}\leq a_1<\frac{2\pi}{5},\; \frac{\pi}{3}\leq a_2<\pi,\; a_1<a_2,
\]
\begin{equation*}
P(x)=\left[\begin{array}{cccc}
p_{1}(x) & p_{2}(x)\\
p_{2}(x) & -p_{1}(x)
\end{array}
\right],\hspace{2mm}
Q(x)=\left[\begin{array}{cccc}
q_{1}(x) & q_{2}(x)\\
q_{2}(x) & -q_{1}(x)
\end{array}
\right],
\\  
\end{equation*}
\\
and $p_{1}(x),p_{2}(x),q_{1}(x), q_{2}(x)\in L^{2}[0,\pi]$ are complex-valued functions such that
\[ P(x)=0, \hspace{2mm} x\in(0,a_{1}), \hspace{2mm} Q(x)=0, \hspace{2mm}  x\in (0,a_{2}).\]
\\
This paper shall answer the question whether the theorem of uniqueness holds or not in the case when the first delay belongs to  $[\frac{\pi}{3},\frac{2\pi}{5})$ and the second one to $[\frac{\pi}{3},\pi).$  In this way, the results from the paper \cite{dju} dealing with Dirac operator with one delay shall be generalized to Dirac operator with two delays. Besides research in inverse problems for delay(s) less than $\frac{\pi}{3}$, further research in this area should also answer the question weather the theorem of uniqueness holds or not  for Sturm-Liouville operators with two delays greater than one-third of the interval. 
\\\\
Hereinafter we will assume that delays $a_{1}$ and $a_{2}$ are known. Also, 
in the following we will assume that $j\in\{1,2\} $ and $m\in\{0,1\}.$
\\\\
Let $\{\lambda_{n,j}^{m}\}$ be the spectra of the BVPs $D_{j}(P,Q,m).$ The inverse problem of recovering  matrix-functions $P(x), x \in (a_{1},\pi)$  and $Q(x), x \in (a_{2},\pi)$ from four spectra has been studied.\\
\\
{\bf Inverse Problem 1.} Given the spectra $\{\lambda_{n,j}^{m}\}$ of the BVPs $D_{j}(P,Q,m)$, find the matrix-functions $P(x)$ and $Q(x)$.\\
\\
The paper is organized as follows: In Section 2 we construct characteristic functions and study asymptotic behavior of eigenvalues. Section 3 is devoted to the solving  Inverse problem 1. We shall show that Theorem of uniqueness holds in the case when delays meet the condition  $2a_1+\frac{a_2}{2}\geq \pi, $ as well as it does not in the case when  $2a_1+\frac{a_2}{2}< \pi.$ 

\section{Spectral properties}
Based on the results from  the paper \cite{vvdju}, we obtain that equation (\ref{a01}) is equivalent to the integral equation 
\begin{equation}\label{a015}
\begin{split}
Y(x,\lambda)=&R(x,\lambda )C+(-1)^{m}\int\limits_{a_{1}}^{x}P(t)S^{-1}(x-t,\lambda)Y(t-a_{1},\lambda)dt\\
&+\int\limits_{a_{2}}^{x}Q(t)S^{-1}(x-t,\lambda)Y(t-a_{2},\lambda)dt
\end{split}
\end{equation}
where 
\[ R(x,\lambda)=\begin{bmatrix}
\cos \lambda x & -\sin\lambda x\\
\sin \lambda x & \cos \lambda x
\end{bmatrix},\;\;\;\;
S(x,\lambda)=\begin{bmatrix}
\sin \lambda x & cos\lambda x\\
-\cos \lambda x & \sin \lambda x
\end{bmatrix}
\]
and 
\[ C=\left[
\begin{array}{cccc}
c_{1} \\
c_{2}
\end{array}
\right]
\]
is a constant vector.
Let 
\[Y^{m}(x,\lambda)= \begin{bmatrix}
 y_{1}^{m}(x,\lambda) \\ 
 y_{2}^{m}(x,\lambda) \end{bmatrix}\\
 \]
be the fundamental vector-solution of equation (\ref{a01}) such that
\begin{equation*}
\begin{split}  
Y^{m}(0,\lambda) = \begin{bmatrix}
 0 \\ 
 -1 \end{bmatrix}.
 \end{split}
 \end{equation*}
From (\ref{a015}) we obtain
\begin{equation}\label{a020}
\begin{split}
Y^{m}(x,\lambda)=&Y_{0}(x,\lambda )+(-1)^{m}\int\limits_{a_{1}}^{x}P(t)S^{-1}(x-t,\lambda)Y^{m}(t-a_{1},\lambda)dt\\
&+\int\limits_{a_{2}}^{x}Q(t)S^{-1}(x-t,\lambda)Y^{m}(t-a_{2},\lambda)dt
\end{split}
\end{equation}
where
\begin{equation*}\
 \begin{split}   
Y_{0}(x,\lambda)= \begin{bmatrix}
 \sin \lambda x \\ 
 -\cos \lambda x \end{bmatrix}.\\
\end{split}
\end{equation*}
\\
We solve integral equation (\ref{a020}) by the method of successive approximation, where representation of fundamental vector-solution depends on the order of delays (see \cite{vvdju}). 
\\\\
Let us  for $k,l \in \{1,2\},$ introduce notations 
\begin{equation*}\
\alpha_{p_{k}p_{l}}^{1}(x)=\int\limits_{x+a_{1}}^{\pi} p_{k}(t)p_{l}(t-x)dt,\;\;\;\;\;\; \alpha_{q_{k}q_{l}}^{2}(x)=\int\limits_{x+a_{2}}^{\pi} q_{k}(t)q_{l}(t-x)dt, 
\end{equation*}
\begin{equation*}\
\begin{split}
&\alpha_{p_{k}q_{l}}^{12}(x)=\int\limits_{x+\frac{a_{1}+a_{2}}{2}}^{\pi} p_{k}(t)q_{l}(t-x-\frac{a_{1}-a_{2}}{2})dt,\\
&\alpha_{q_{k}p_{l}}^{12}(x)=\int\limits_{x+\frac{a_{1}+a_{2}}{2}}^{\pi} q_{k}(t)p_{l}(t-x-\frac{a_{2}-a_{1}}{2})dt.\
\end{split}
\end{equation*}
\\
Eigenvalues of the BVPs $D_{j}(P,Q,m)$ coincide with zeros of the entire function
\begin{equation*}\
\begin{split}    
\Delta_{j}^{m}(\lambda)=y_{j}^{m}(\pi,\lambda)
\end{split}
\end{equation*}
which is called {\it characteristic function} of BVPs $D_{j}(P,Q,m)$. It has been shown in  \cite{vvdju} that characteristic functions of BVPs $D_{j}(P,Q,m)$ can be represented in the form 
\begin{equation}\label{14}
\begin{split}   
\Delta_{1}^{m}(\lambda)=\sin \lambda \pi + \int\limits_{\frac{a_{1}}{2}}^{\pi-\frac{a_{1}}{2}} K^{m}(x) \cos \lambda (\pi-2x)dx -\int\limits_{\frac{a_{1}}{2}}^{\pi-\frac{a_{1}}{2}} G^{m}(x) \sin \lambda (\pi-2x)dx,
\end{split}
\end{equation}
\begin{equation}\label{15}
\Delta_{2}^{m}(\lambda)=-\cos \lambda \pi +\int\limits_{\frac{a_{1}}{2}}^{\pi-\frac{a_{1}}{2}} K^{m}(x) \sin \lambda (\pi-2x)dx + 
\int\limits_{\frac{a_{1}}{2}}^{\pi-\frac{a_{1}}{2}} G^{m}(x) \cos \lambda (\pi-2x)dx
\end{equation}
where 
\begin{align}\label{14s}
K^{m}(x)=K_{1}^{m}(x)+K_{2}(x)+K_{12}^{m}(x)+K_{21}^{m}(x)
\end{align}
\begin{align}\label{14t}
G^{m}(x)=G_{1}^{m}(x)+G_{2}(x)+G_{12}^{m}(x)+G_{21}^{m}(x)
\end{align}
and for $x\in (\frac{a_{1}}{2},\pi-\frac{a_{1}}{2})$
\begin{equation*}
\begin{split}
K_{1}(x)= (-1)^{m}p_{1}(x+\frac{a_{1}}{2})
-\bigg(\alpha_{p_{1}p_{2}}^{1}(x) -
\alpha_{p_{2}p_{1}}^{1}(x) \bigg)
 \mathbbm{1}_{(a_{1},\pi-a_{1})} (x),\\
G_{1}(x)= (-1)^{m}p_{2}(x+\frac{a_{1}}{2})
-\bigg(\alpha_{p_{1}p_{1}}^{1}(x) +
\alpha_{p_{2}p_{2}}^{1}(x)\bigg) \mathbbm{1}_{(a_{1},\pi-a_{1})} (x),
\end{split}
\end{equation*}
for  $x\in (\frac{a_{2}}{2},\pi-\frac{a_{2}}{2})$
\begin{equation*}
\begin{split}
K_{2}(x)=q_{1}(x+\frac{a_{2}}{2})
-\bigg( \alpha_{q_{1}q_{2}}^{2}(x) - \alpha_{q_{2}q_{1}}^{2}(x) \bigg) \mathbbm{1}_{(a_{2},\pi-a_{2})}(x),\\
G_{2}(x)=q_{2}(x+\frac{a_{2}}{2})
-\bigg(\alpha_{q_{1}q_{1}}^{1}(x) +
\alpha_{q_{2}q_{q}}^{2}(x)\bigg) \mathbbm{1}_{(a_{2},\pi-a_{2})}(x),
\end{split}
\end{equation*}
and for $x\in (\frac{a_{1}+a_{2}}{2},\pi-\frac{a_{1}+a_{2}}{2})$
\begin{equation*}
\begin{split}
K_{12}^{m}(x)=
(-1)^{m}\alpha_{p_{2}q_{1}}^{12}(x)-(-1)^{m}\alpha_{p_{1}q_{2}}^{12}(x),\\\\
K_{21}^{m}(x)=
(-1)^{m}\alpha_{q_{2}p_{1}}^{12}(x)-(-1)^{m}\alpha_{q_{1}p_{2}}^{12}(x),
\end{split}
\end{equation*}
\begin{equation*}
\begin{split}
G_{12}^{m}(x)=-(-1)^{m}\alpha_{p_{1}q_{1}}^{12}(x)-(-1)^{m}\alpha_{p_{2}q_{2}}^{12}(x),
\end{split}
\end{equation*}
\begin{equation*}
\begin{split}
G_{21}^{m}(x)=- (-1)^{m}\alpha_{q_{1}p_{1}}^{12}(x)-(-1)^{m}\alpha_{q_{2}p_{2}}^{12}(x).
\end{split}
\end{equation*}
\\
Now we consider the asymptotic behavior of eigenvalues of BVPs $D_{j}(P,Q,m).$ Using the standard approach involving Rouch\'e's theorem or proof from \cite{dju}, one can show that the next theorem holds.
\begin{theorem}\label{th1}
The boundary value problems $D_{j}(P,Q,m)$ have infinitely many eigenvalues $\lambda_{n,j}^{m},$ $n\in{\mathbb Z},$ of the form
\begin{equation*}
\lambda_{n,j}^{m}=n+\frac{1-j}2+\varkappa_{n,j}^{m} 
\end{equation*}
where for $\varkappa_{n,j}^{m}\ne{0}$ and for  $|n| \to \infty $
\begin{equation*}
\begin{split}
\varkappa_{n,1}^{m} =& \frac{1}{\pi} \int\limits_{\frac{a_{1}}{2}}^{\pi-\frac{a_{1}}{2}} K^{m}(x) \sin2nx \;dx + \frac{1}{\pi} \int\limits_{\frac{a_{1}}{2}}^{\pi-\frac{a_{1}}{2}} G^{m}(x) \cos2nx \;dx+ o(\varkappa_{n,1}^{m}),\\
\varkappa_{n,2}^{m} =& \frac{1}{\pi} \int\limits_{\frac{a_{1}}{2}}^{\pi-\frac{a_{1}}{2}} K^{m}(x) \sin(2n-1)x \; dx \\
&+ \frac{1}{\pi} \int\limits_{\frac{a_{1}}{2}}^{\pi-\frac{a_{1}}{2}} G^{m}(x) \cos(2n-1)x \; dx+ o(\varkappa_{n,2}^{m}).\\
\end{split}
\end{equation*}
\end{theorem}

\section{Recovering of the matrix-functions}\label{sec3}

In order to recover the matrix-functions $P(x)$ and $Q(x)$ from the spectra $\{\lambda_{n,j}^{m}\}$, at the beginning we construct characteristic functions by Hadamar theorem of factorization.   
\begin{lemma}
The specification of the spectra $(\lambda_{n,j}^{m})$ uniquely determines the characteristic functions $\Delta_{1}^{m}$ and $\Delta_{2}^{m}$ of BVPs $D_{j}(P,Q,m)$ by the formulas
\\
\begin{equation*}
\begin{split}
&\Delta_{1}^{m}(\lambda)=\pi(\lambda_{0,1}^{m}-\lambda)\prod\limits_{|n|\in N}\frac{\lambda_{n,1}^{m}-\lambda}{n}e^{\frac{\lambda}{n}},\\\\
&\Delta_{2}^{m}(\lambda)=\prod\limits_{n\in Z}\frac{\lambda_{n,2}^{m}-\lambda}{n-\frac{1}{2}}e^{\frac{\lambda}{n-\frac{1}{2}}}.
\end{split}
\end{equation*}
\end{lemma}
\begin{proof}
See Theorem 5 in \cite{buttt}.
\end{proof}
\clearpage
\begin{wrapfigure}{R}{0.30\textwidth}
\centering
\includegraphics[width=0.22\textwidth]{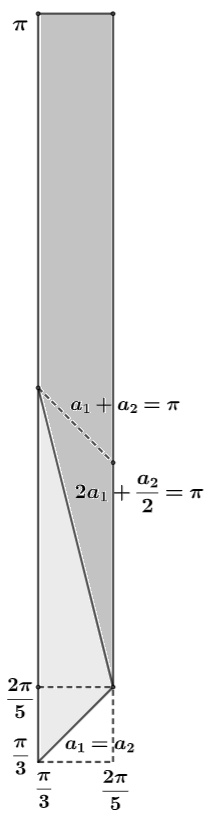}
\caption{Picture 3.1}
\end{wrapfigure}
Using approach from the paper  \cite{dju}, we can recover functions $K^{m}(x)$ and $G^{m}(x)$ by formulas
\begin{equation}\label{20i}
\begin{split}
K^{m}(x)= \sum_{n \in \mathbb{Z}}(\frac{(-1)^n}{\pi}\Theta_{1}^m(n) +i\frac{(-1)^n}{\pi}\Theta_{2}^m(n)) e^{2inx} \\
\end{split}
\end{equation}
\begin{equation}\label{20j}
\begin{split}
G^{m}(x)= \sum_{n \in \mathbb{Z}}(\frac{(-1)^n}{\pi}\Theta_{3}^m(n) +i\frac{(-1)^n}{\pi}\Theta_{4}^m(n)) e^{2inx} 
\end{split}
\end{equation}
where
\begin{equation*}\
\begin{split}
&\Theta_{1}^m(\lambda)=\frac{\Delta_{1}^{m}(\lambda)+\Delta_{1}^{m}(-\lambda)}{2},\\\\
&\Theta_{2}^m(\lambda)=\frac{\Delta_{2}^{m}(\lambda)-\Delta_{2}^{m}(-\lambda)}{2},\\\\
&\Theta_{3}^m(\lambda)=\frac{-\Delta_{1}^{m}(\lambda)+\Delta_{1}^{m}(-\lambda)}{2}+\sin{\lambda \pi},\\\\
&\Theta_{4}^m(\lambda)=\frac{\Delta_{2}^{m}(\lambda)+\Delta_{2}^{m}(-\lambda)}{2}+\cos{\lambda \pi}.\\
\end{split}
\end{equation*}
Now we come to our main result.  Using functions $K^m(x)$ and $G^m(x)$ from  (\ref{20i}) and (\ref{20j}) respectively,  we shall answer the question  whether the theorem of uniqueness for Inverse problem 1 holds on the set
\[R=\{(a_1,a_2):\;\frac{\pi}{3}<a_1<\frac{2\pi}{5}\;\wedge\; \frac{\pi}{3}<a_2<\pi \;\wedge \;a_1<a_2\}.\] 
It will be shown that it is true on the subset 
\[R_1=\{(a_1,a_2)\in R: \;\frac{\pi}{3}<a_1<\frac{2\pi}{5}<a_2<\pi\;\wedge\; 2a_1+\frac{a_2}{2}\geq\pi\},\]  
while on the subset
\[R_2=R\setminus R_1=\{(a_1,a_2) \in R: \frac{\pi}{3}<a_1<\frac{2\pi}{5}\;\wedge\;\frac{\pi}{3}<a_2<\frac{2\pi}{3}\;\wedge \;2a_1+\frac{a_2}{2}<\pi\}\]  
that is not true, Picture 3.1.
\clearpage
Firstly we will prove that the theorem of uniqueness holds on the subset $R_1.$  
We will recover functions $p_{1}(x),\hspace {1mm} p_{2}(x)$ by formulas
\begin{equation}\label{23}
\begin{split}
p_{1}(x)=\frac{1}{2}(K^{0}(x-\frac{a_{1}}{2})-K^{1}(x-\frac{a_{1}}{2}))+\frac{1}{2}A_1(x-\frac{a_{1}}{2}),\\
p_{2}(x)=\frac{1}{2}(G^{0}(x-\frac{a_{1}}{2})-G^{1}(x-\frac{a_{1}}{2}))+\frac{1}{2}A_2(x-\frac{a_{1}}{2})
\end{split}
\end{equation}
where for $x\in(\frac{a_1+a_2}{2},\pi-\frac{a_1+a_2}{2})$
\begin{equation*}\
\begin{split}
&A_1(x)=-\alpha_{p_{2}q_{1}}^{12}(x)+
\alpha_{p_{1}q_{2}}^{12}(x)-
\alpha_{q_{2}p_{1}}^{12}(x)+
\alpha_{q_{1}p_{2}}^{12}(x),\\\\
&A_2(x)=\alpha_{p_{1}q_{1}}^{12}(x)+
\alpha_{p_{2}q_{2}}^{12}(x)+
\alpha_{q_{1}p_{1}}^{12}(x)+
\alpha_{q_{2}p_{2}}^{12}(x)
\end{split}
\end{equation*}
and 
\[A_1(x)=A_2(x)=0,\;\;\; x\notin (\frac{a_1+a_2}{2},\pi-\frac{a_1+a_2}{2}).\]
Functions $q_{1}(x),\hspace {1mm} q_{2}(x)$ will be recovered by formulas
\begin{equation}\label{24}
\begin{split}
q_{1}(x)=\frac{1}{2}(K^{0}(x-\frac{a_{2}}{2})+K^{1}(x-\frac{a_{2}}{2}))+\frac{1}{2}B_1(x-\frac{a_{2}}{2}),\\
q_{2}(x)=\frac{1}{2}(G^{0}(x-\frac{a_{2}}{2})+G^{1}(x-\frac{a_{2}}{2}))+\frac{1}{2}B_2(x-\frac{a_{2}}{2})
\end{split}
\end{equation}
where for $x\in (a_1,a_2)\cup(\pi-a_2,\pi-a_1)$
\begin{equation*}
\begin{split}
&B_1(x)=
\alpha_{p_{1}p_{2}}^{1}(x)-
\alpha_{p_{2}p_{1}}^{1}(x),\\\\
&B_2(x)=\alpha_{p_{1}p_{1}}^{1}(x)+
\alpha_{p_{2}p_{2}}^{1}(x),
\end{split}
\end{equation*}
for $x\in (a_2,\pi-a_2)$
\begin{equation*}\
\begin{split}
&B_1(x)=\alpha_{p_{1}p_{2}}^{1}(x)-
\alpha_{p_{2}p_{1}}^{1}(x)+
\alpha_{q_{1}q_{2}}^{2}(x)-
\alpha_{q_{2}q_{1}}^{2}(x),
 \\\\
&B_2(x)=\alpha_{p_{1}p_{1}}^{1}(x)+
\alpha_{p_{2}p_{1}}^{1}(x)+
\alpha_{q_{1}q_{1}}^{2}(x)+
\alpha_{q_{2}q_{2}}^{2}(x),
\end{split}
\end{equation*}
and 
\[B_1(x)=B_2(x)=0,\;\;\; x\notin (a_1,\pi-a_1).\]

\begin{theorem}
Let  $\frac{\pi}{3}\leq a_1<\frac{2\pi}{5}<a_2<\pi$ and $2a_1+\frac{a_2}{2}\geq \pi.$ The spectra $(\lambda_{n,j}^{m})$  of BVPs $D_{j}(P,Q,m) $ uniquely determine matrix-functions $P(x), \hspace{1mm} x \in (a_1,\pi)$ and $Q(x), \hspace{1mm} x \in (a_2,\pi).$ \
\end{theorem}
\begin{proof} We distinguish cases when the sum of delays is greater or less than $\pi$. 
\\\\
1. Let $a_1+a_2\geq \pi.$ Then we have 
\[\frac{\pi}{3}\leq a_1<\frac{2\pi}{5}<a_2<2a_1<\pi\]
or 
\[\frac{\pi}{3}\leq a_1<\frac{2\pi}{5}<2a_1<a_2<\pi.\]
We will prove the theorem for the first case, i.e. assuming that $a_2<2a_1,$ while the proof for the second case differs only in the order of the intervals on which functions will be determined. From (\ref{14s}) and (\ref{14t}),  for $x\in (\frac{a_{1}}{2},\frac{a_{2}}{2})\cup (\pi-\frac{a_{2}}{2},\pi-\frac{a_{1}}{2}),$  we have  
\begin{equation*}
\begin{split}
p_{1}(x)=K^{0}(x-\frac{a_{1}}{2}), \;\;\;\;\;p_{2}(x)=G^{0}(x-\frac{a_{1}}{2})
\end{split}    
\end{equation*}
i.e. we determine functions
\begin{equation*}
\begin{split}
p_{1}(x),\hspace {4mm} p_{2}(x),
\hspace {2mm}
x\in (a_{1},\frac{a_{1}+a_{2}}{2})\cup (\pi-\frac{a_{2}}{2}+\frac{a_{1}}{2},\pi).
\end{split}
\end{equation*}
If  $  x\in (\frac{a_{2}}{2},a_1)\cup(\pi-a_1,\pi-\frac{a_{2}}{2}),$  functions $K^m(x)$ and $G^m(x)$ have the form 
\begin{equation*}\
\begin{split}
&K^{m}(x)={(-1)}^{m}p_{1}(x+\frac{a_{1}}{2})+q_{1}(x+\frac{a_{2}}{2})
\end{split}
\end{equation*}
and
\begin{equation*}\
\begin{split}
&G^{m}(x)={(-1)}^{m}p_{2}(x+\frac{a_{1}}{2})+q_{2}(x+\frac{a_{2}}{2}).
\end{split} 
\end{equation*}
Then we recover functions 
\begin{equation*}
\begin{split}
p_{1}(x),\hspace {4mm} p_{2}(x),
\hspace {2mm}
x\in (\frac{a_{1}+a_{2}}{2},\frac{3a_1}{2})\cup(\pi-\frac{a_{1}}{2}, \pi-\frac{a_{2}}{2}+\frac{a_{1}}{2})
\end{split}
\end{equation*}
by formulas  (\ref{23}) for $A_1(x)=A_2(x)=0$, as well as functions
\begin{equation*}
\begin{split}
q_{1}(x),\hspace {4mm} q_{2}(x),
\hspace {2mm}
x\in (a_{2},a_1+\frac{a_2}{2})\cup(\pi-a_1+\frac{a_{2}}{2},\pi)
\end{split}
\end{equation*}
by formulas (\ref{24}) for $B_1(x)=B_2(x)=0.$ Finally, from (\ref{14s}) and (\ref{14t}) we obtain that for $x\in ({a_{1}},\pi-{a_{1}})$ functions $K^m(x)$ and $G^m(x)$ have the form 
\begin{equation}\label{57a}
\begin{split}
&K^{m}(x)={(-1)}^{m}p_{1}(x+\frac{a_{1}}{2})+q_{1}(x+\frac{a_{2}}{2})-
\alpha_{p_{1}p_{2}}^{1}(x)+
\alpha_{p_{2}p_{1}}^{1}(x)  
\end{split}
\end{equation}
\begin{equation}\label{57b}
\begin{split}
&G^{m}(x)={(-1)}^{m}p_{2}(x+\frac{a_{1}}{2})+q_{2}(x+\frac{a_{2}}{2})-
\alpha_{p_{1}p_{1}}^{1}(x)-
\alpha_{p_{2}p_{2}}^{1}(x).
\end{split}   
\end{equation}
Then we recover functions 
\begin{equation*}
\begin{split}
p_{1}(x),\hspace {4mm} p_{2}(x),
\hspace {2mm}
x\in (\frac{3a_{1}}{2},\pi-\frac{a_{1}}{2})
\end{split}
\end{equation*}
by formulas (\ref{23}) for $A_1(x)=A_2(x)=0.$
In this way functions  $p_{1}(x)$ and $p_{2}(x)$ are recovered on  $(a_{1},\pi).$ Then integrals  $\alpha_{p_{k}p_{l}}^{1}(x), k,l\in\{1,2\}$ are known, too. Now,  using formulas (\ref{24}) for $B_1(x)=\alpha_{p_{1}p_{2}}^{1}(x)-
\alpha_{p_{2}p_{1}}^{1}(x)$ and $B_2(x)=\alpha_{p_{1}p_{1}}^{1}(x)+
\alpha_{p_{2}p_{2}}^{1}(x),$ we determine functions
\begin{equation*}
\begin{split}
q_{1}(x),\hspace {4mm} q_{2}(x),
\hspace {2mm}
x\in (a_1+\frac{a_2}{2},\pi-a_1+\frac{a_{2}}{2}),
\end{split}
\end{equation*}
so they are also completely recovered on $(a_2,\pi).$
\\\\
2. Let us now consider the case $a_{1}+a_{2}<\pi.$ Taking  the definition of the subset $R_1$ into account, we have 
\[\frac{2\pi}{5}<a_2<\frac{2\pi}{3}<2a_1. \footnote {The order of delays $2a_{1}<a_{2}<a_{1}+a_{2}<\pi$ is not possible since $2a_{1}<a_{2}\implies \pi<3a_{1}<a_{1}+a_{2} $}\] 
At the beginning, in the same way as in the proof of previous case, for $x\in (\frac{a_{1}}{2},\frac{a_{2}}{2})\cup (\pi-\frac{a_{2}}{2},\pi-\frac{a_{1}}{2})$ and  $  x\in (\frac{a_{2}}{2},a_1)\cup(\pi-a_1,\pi-\frac{a_{2}}{2}),$  using formulas (\ref{23}) and (\ref{24}),  we determine functions 
\begin{equation*}
\begin{split}
p_{1}(x),\hspace {4mm} p_{2}(x),
\hspace {2mm}
x\in (a_{1},\frac{3a_{1}}{2})\cup (\pi-\frac{a_{1}}{2},\pi)
\end{split}
\end{equation*}
and functions
\begin{equation*}
\begin{split}
q_{1}(x),\hspace {4mm} q_{2}(x),
\hspace {2mm}
x\in (a_2,a_{1}+\frac{a_{2}}{2})\cup (\pi-a_1+\frac{a_{2}}{2},\pi). 
\end{split}
\end{equation*}
For $x\in (a_1,\frac{a_{1}+a_2}{2})\cup (\pi-\frac{a_{1}+a_2}{2},\pi-a_1)$ functions $K^m(x)$ and $G^m(x)$ have the form (\ref{57a}) and (\ref{57b})  and we determine functions 
\begin{equation*}
\begin{split}
p_{1}(x),\hspace {4mm} p_{2}(x),
\hspace {2mm}
x\in (\frac{3a_{1}}{2},a_1+\frac{a_{2}}{2})\cup(\pi-\frac{a_{2}}{2}, \pi-\frac{a_{1}}{2})
\end{split}
\end{equation*}
by formulas (\ref{23}) for $A_1(x)=A_2(x)=0.$ In order to determine functions $q_1(x), q_2(x)$, it is needed to show that integrals $\alpha_{p_{k}p_{l}}^{1}(x), k,l\in \{1,2\},$ are known. For arguments of subintegral functions $p_1(x), p_2(x)$ is valid
\[t>x+a_{1}>2a_1 > \pi-\frac{a_{2}}{2}, \]
\[t-x<\pi-a_1<a_1+\frac{a_{2}}{2} \]
due to the assumption 
\[2a_1+\frac{a_2}{2}\geq \pi.\]
Therefore, arguments of subintegral functions $p_{1}(x), p_{2}(x)$ belong to the interval 
$ (a_{1},a_1+\frac{a_{2}}{2})\cup (\pi-\frac{a_{2}}{2},\pi),$ hence integrals $\alpha_{p_{k}p_{l}}^{1}(x) $  are known.Then, from (\ref{24}) for $B_1(x)=\alpha_{p_{1}p_{2}}^{1}(x)-
\alpha_{p_{2}p_{1}}^{1}(x)$ and $B_2(x)=\alpha_{p_{1}p_{1}}^{1}(x)+
\alpha_{p_{2}p_{2}}^{1}(x),$ we can determine functions 
\begin{equation*}
\begin{split}
q_{1}(x),\hspace {4mm} q_{2}(x),
\hspace {2mm}
x\in (a_1+\frac{a_2}{2}, a_2+\frac{a_1}{2})\cup(\pi-\frac{a_{1}}{2}, \pi-a_1+\frac{a_{2}}{2}).
\end{split}
\end{equation*}
In the following we have in mind that the functions $p_1(x)$ and $p_2(x)$ are recovered on the interval $(a_1, a_1+\frac{a_2}{2})\cup(\pi-\frac{a_2}{2}, \pi)$ and  functions $q_1(x)$ and $q_2(x)$ on the interval $(a_2, a_2+\frac{a_1}{2})\cup(\pi-\frac{a_1}{2}, \pi). $ We differ two cases, depending on the condition whether the second delay is grater or less of $\frac{\pi}{2}.$ 
\\\\
2.1.  Let $\frac{\pi}{2}\leq a_2<\frac{2\pi}{3}.$ Then it remains to recover matrix-functions $P(x)$ and $Q(x)$ on the interval $  (\frac{a_1+a_{2}}{2},\pi-\frac{a_1+a_{2}}{2})$. We have
\begin{equation}\label{700}
\begin{split}
&K^{m}(x)={(-1)}^{m}p_{1}(x+\frac{a_{1}}{2})+q_{1}(x+\frac{a_{2}}{2})-
\alpha_{p_{1}p_{2}}^{1}(x)+
\alpha_{p_{2}p_{1}}^{1}(x)\\
&+(-1)^{m}\alpha_{p_{2}q_{1}}^{12}(x)-
(-1)^{m}\alpha_{p_{1}q_{2}}^{12}(x)+
(-1)^{m}\alpha_{q_{2}p_{1}}^{12}(x)-
(-1)^{m}\alpha_{q_{1}p_{2}}^{12}(x)   
\end{split}
\end{equation}
and
\begin{equation}\label{770}
\begin{split}
&G^{m}(x)={(-1)}^{m}p_{2}(x+\frac{a_{1}}{2})+q_{2}(x+\frac{a_{2}}{2})-
\alpha_{p_{1}p_{1}}^{1}(x)-
\alpha_{p_{2}p_{2}}^{1}(x)\\
&-(-1)^{m}\alpha_{p_{1}q_{1}}^{12}(x)-
(-1)^{m}\alpha_{p_{2}q_{2}}^{12}(x)-
(-1)^{m}\alpha_{q_{1}p_{1}}^{12}(x)-
(-1)^{m}\alpha_{q_{2}p_{2}}^{12}(x).
\end{split}   
\end{equation}
One can easily obtain that integrals $\alpha_{p_{k}p_{l}}^{1}(x)$ are known. We have 
\[t>x+a_{1}>\frac{a_{1}+a_{2}}{2}+a_{1}> \pi-\frac{a_{2}}{2}, \]
\[t-x<\pi-\frac{a_{1}+a_{2}}{2}<a_1+\frac{a_{2}}{2} \]
since
\[\frac{3a_1}{2}+a_2>2a_1+\frac{a_2}{2}>\pi.\]
Then, from (\ref{24})  we can determine functions 
\begin{equation*}
\begin{split}
q_{1}(x),\hspace {4mm} q_{2}(x),
\hspace {2mm}
x\in (a_2+\frac{a_{1}}{2},\pi-\frac{a_{1}}{2}).
\end{split}
\end{equation*}
In that way functions $q_{1}(x),\hspace {1mm} q_{2}(x)$ are completely recovered on $(a_2,\pi).$ Now it is not difficult to show that "mixed" integrals $\alpha_{p_{k}q_{l}}^{12}(x)$  and $\alpha_{q_{k}p_{l}}^{12}(x) $ are also known. Indeed, for arguments of subintegral functions $p_1(x), p_2(x)$ in "mixed" integrals is valid 
\[t>x+\frac{a_{1}+a_{2}}{2}>a_{1}+a_2> \pi-\frac{a_{2}}{2}, \]
\[t-x-\frac{a_{2}-a_{1}}{2}<\pi-\frac{a_{1}+a_{2}}{2}-\frac{a_{2}-a_{1}}{2}=\pi-a_2<a_1+\frac{a_{2}}{2} \]
since
\[a_{1}+\frac{3a_{2}}{2}>2a_1+\frac{a_2}{2}>\pi. \]
Therefore, arguments of subintegral functions $p_{1}(x), p_{2}(x)$ belong to the interval 
$ (a_{1},a_1+\frac{a_{2}}{2})\cup (\pi-\frac{a_{2}}{2},\pi)$ and "mixed" integrals are known.
Then, using formulas (\ref{23}),  we can determine functions 
\begin{equation*}
\begin{split}
p_{1}(x),\hspace {4mm} p_{2}(x),
\hspace {2mm}
x\in (a_1+\frac{a_2}{2},\pi-\frac{a_{2}}{2})
\end{split}
\end{equation*}
so they are also completely recovered on $(a_1,\pi).$
\\\\
2.2.  Let $\frac{2\pi}{5}\leq a_2<\frac{\pi}{2}.$  It remains to show that theorem of uniqueness is valid on the intervals $ (\frac{a_1+a_{2}}{2},a_2)\cup (\pi-a_2, \pi-\frac{a_1+a_{2}}{2})$ and $(a_2,\pi-a_2). $ 
On the interval $  (\frac{a_1+a_{2}}{2},a_2)\cup (\pi-a_2, \pi-\frac{a_1+a_{2}}{2})$ functions $K^{m}(x)$ and  $G^{m}(x)$ have the form (\ref{700}) and (\ref{770}) respectively.  In the same way as in the proof for the case 2.1, one can show that integrals $\alpha_{p_{k}p_{l}}^{1}(x)$ are known, so we can determine functions 
\begin{equation*}
\begin{split}
q_{1}(x),\hspace {4mm} q_{2}(x),
\hspace {2mm}
x\in (a_2+\frac{a_{1}}{2},\frac{3a_2}{2})\cup(\pi-\frac{a_{2}}{2}, \pi-\frac{a_{1}}{2}) 
\end{split}
\end{equation*}
by formulas (\ref{24}). Now we will show that "mixed" integrals $\alpha_{p_{k}q_{l}}^{12}(x)$  and $\alpha_{q_{k}p_{l}}^{12}(x)$ are also known. In the same way as in previous case we show that arguments of subintegral functions $p_{1}(x), p_{2}(x)$ belong to the interval 
$ (a_{1},a_1+\frac{a_{2}}{2})\cup (\pi-\frac{a_{2}}{2},\pi)$. For arguments of subintegral functions $q_{1}(x), q_{2}(x)$  in "mixed" integrals we have
\[t>x+\frac{a_{1}+a_{2}}{2}>a_1+a_2> \pi-\frac{a_{2}}{2}, \]
\[t-x-\frac{a_{1}-a_{2}}{2}<\pi-\frac{a_{1}+a_{2}}{2}-\frac{a_{1}-a_{2}}{2}=\pi-a_1<\frac{3a_{2}}{2} \]
since
\[a_{1}+\frac{3a_{2}}{2}>2a_1+\frac{a_2}{2}> \pi. \]
Then, using formulas (\ref{23}), we can determine functions 
\begin{equation*}
\begin{split}
p_{1}(x),\hspace {4mm} p_{2}(x),
\hspace {2mm}
x\in (a_1+\frac{a_2}{2},a_2+\frac{a_1}{2})\cup(\pi-a_2+\frac{a_{1}}{2},\pi-\frac{a_{2}}{2}).
\end{split}
\end{equation*}
Let us finally consider the interval $ (a_2,\pi-a_2)$. Then functions $K^{m}(x), G^{m}(x)$ have the form
\begin{equation}\label{788}
\begin{split}
K^{m}(x)=&{(-1)}^{m}p_{1}(x+\frac{a_{1}}{2})+q_{1}(x+\frac{a_{2}}{2})-
\alpha_{p_{1}p_{2}}^{1}(x)\\
&+\alpha_{p_{2}p_{1}}^{1}(x)-\alpha_{q_{1}q_{2}}^{2}(x)+\alpha_{q_{2}q_{1}}^{2}(x)+
(-1)^{m}\alpha_{p_{2}q_{1}}^{12}(x)\\
&-(-1)^{m}\alpha_{p_{1}q_{2}}^{12}(x)+
(-1)^{m}\alpha_{q_{2}p_{1}}^{12}(x)-
(-1)^{m}\alpha_{q_{1}p_{2}}^{12}(x)
\end{split}
\end{equation}
and
\begin{equation}\label{789}
\begin{split}
G^{m}(x)=&{(-1)}^{m}p_{2}(x+\frac{a_{1}}{2})+q_{2}(x+\frac{a_{2}}{2})-
\alpha_{p_{1}p_{1}}^{1}(x)\\
&-\alpha_{p_{2}p_{2}}^{1}(x)-\alpha_{q_{1}q_{1}}^{2}(x)-\alpha_{q_{2}q_{2}}^{2}(x)-(-1)^{m}\alpha_{p_{1}q_{1}}^{12}(x)\\
&-(-1)^{m}\alpha_{p_{2}q_{2}}^{12}(x)-
(-1)^{m}\alpha_{q_{1}p_{1}}^{12}(x)-
(-1)^{m}\alpha_{q_{2}p_{2}}^{12}(x).
\end{split}   
\end{equation}
Let us show that integrals $\alpha_{p_{k}p_{l}}^{1}(x),$ as well as integrals $\alpha_{q_{k}q_{l}}^{2}(x),$ are known. For arguments of subintegral functions $p_1(x), p_2(x)$ we have
\[t>x+a_1>a_1+a_2> \pi-a_2 +\frac{a_{1}}{2}, \]
\[t-x<\pi-a_2<a_2+\frac{a_{1}}{2} \]
since
\[2a_2+\frac{a_1}{2}>2a_1+\frac{a_2}{2}> \pi. \]
For arguments of subintegral functions $q_1(x), q_2(x)$ is valid
\[t>x+a_2>2a_2> \pi-\frac{a_{2}}{2}, \]
\[t-x<\pi-a_2<\frac{3a_{2}}{2} \]
since
\[\frac{5a_2}{2}>\pi. \]
In that way,  using formulas (\ref{24}), we can recover functions 
\begin{equation*}
\begin{split}
q_{1}(x),\hspace {4mm} q_{2}(x),
\hspace {2mm}
x\in (\frac{3a_{2}}{2},\pi-\frac{a_{2}}{2})
\end{split}
\end{equation*}
so they are completely recovered on $(a_2,\pi)$.  It remains to  show that "mixed" integrals $\alpha_{p_{k}q_{l}}^{1}(x)$  and $\alpha_{q_{k}p_{l}}^{1}(x)$ are also known.  Due to the condition $\frac{5a_2}{2}>\pi,$ for arguments of subintegral functions $p_1(x), p_2(x)$ is valid 
\[t>x+\frac{a_{1}+a_{2}}{2}>a_2+\frac{a_{1}+a_{2}}{2}> \pi-a_2+\frac{a_{1}}{2}, \]
\[t-x-\frac{a_{2}-a_{1}}{2}<\pi-a_2-\frac{a_{2}-a_{1}}{2}<a_2+\frac{a_{1}}{2}. \]
Then we can determine functions 
\begin{equation*}
\begin{split}
p_{1}(x),\hspace {4mm} p_{2}(x),
\hspace {2mm}
x\in (a_2+\frac{a_{1}}{2},\pi-a_2+\frac{a_{1}}{2}) 
\end{split}
\end{equation*}
by formulas (\ref{23}), and they are also completely recovered on $(a_1,\pi).$ 
Theorem is proved.
\end{proof}
Now we will show that theorem of uniqueness does not hold on the subset $R_2.$
For that purpose let us for fixed $a_{1},a_{2}$ and $x\in (a_{1}+\frac{a_2}{2},\pi-a_{1})$ define the integral operator 
\begin{equation}\label{400}
\begin{split}
M(f(x))=\int_{a_{1}+\frac{a_2}{2}}^{\pi-x+\frac{a_{2}}{2}} f(t) h(t+x-\frac{a_{2}}{2}) dt \;\;
\end{split}
\end{equation}
for some non-zero real function $h(x) \in L_{2}(2a_1+\frac{a_2}{2},\pi).$ Operator $M(f(x))$ is self-adjoint since 
\[\int_{a_{1}+\frac{a_2}{2}}^{\pi-a_1}M(f(x))g(x)dx=
\int_{a_{1}+\frac{a_2}{2}}^{\pi-a_1}\bigg(\int_{a_{1}+\frac{a_2}{2}}^{\pi-x+\frac{a_{2}}{2}} f(t) h(t+x-\frac{a_{2}}{2}) dt\bigg)g(x)dx\]
\[=\int_{a_{1}+\frac{a_2}{2}}^{\pi-a_1}f(t)\bigg(\int_{a_{1}+\frac{a_2}{2}}^{\pi-t-\frac{a_{2}}{2}}  h(t+x-\frac{a_{2}}{2})g(x)dx\bigg)dt=\int_{a_{1}+\frac{a_2}{2}}^{\pi-a_1}f(t)M(g(t))dt.\]
\\
We can choose function $h(x)$ such that the operator $M(f(x))$ has eigenvalue $\eta_{1}=1$  with corresponding eigenfunction $e_1(x)$ (see \cite{DB21}), i.e.
\begin{equation}\label{301}
\begin{split}
M(e_1(x))=e_1(x),\;\; x\in (a_{1}+\frac{a_2}{2},\pi-a_{1}).
\end{split}
\end{equation}
We construct the family of functions
\begin{equation*}
\begin{split}
D_\beta=\{p_1(x), p_2^\beta(x), q_1(x),q_2^\beta(x):\beta\in C\}
\end{split}
\end{equation*}
where
\begin{equation*}\
p_1(x)=0,\;x\in (0,\pi),
\end{equation*}
\\
\begin{equation}\label{652}
p_2^\beta(x)=\left\{\begin{array}{rl} 0,\;\;  x\in (0,a_{1}+\frac{a_2}{2})\cup (\pi-a_1,2a_1+\frac{a_2}{2})\\ 
\beta e_{1}(x),\;\;  x\in (a_{1}+\frac{a_2}{2},\pi-a_1) \\
h(x),\;\; x\in (2a_1+\frac{a_2}{2},\pi)\end{array}\right.   
\end{equation} 
\\
\begin{equation*}\
q_1(x)=0, \;\;\;x\in (0,\pi), 
\end{equation*}
\\
\begin{equation}\label{654}
q_2^\beta(x)=\left\{\begin{array}{rl} 0,\;\;  x\in (0,a_{1}+\frac{a_2}{2})\cup (\pi-a_1,\pi)\\ 
\beta e_{1}(x),\;\;  x\in (a_{1}+\frac{a_2}{2},\pi-a_1). \end{array}\right.   
\end{equation} 
\\
Using this family of functions, we will prove that the solution of Inverse problem 1 is not unique if $2a_1+\frac{a_2}{2}<\pi.$ Let us for this purpose denote 
\[P_{\beta }(x)=\left[\begin{array}{cccc}
0 & p_{2}^\beta(x)\\
p_{2}^\beta(x) & 0
\end{array}
\right],\hspace{2mm}
Q_\beta(x)=\left[\begin{array}{cccc}
0 & q_2^\beta(x)\\
q_2^\beta(x) & 0
\end{array}
\right].\]
\begin{theorem} Let $\frac{\pi}{3}\leq a_1<\frac{2\pi}{5},\; \frac{\pi}{3}\leq a_{2} <\frac{2\pi}{3},\; a_1<a_2$ and  $2a_1+\frac{a_2}{2}<\pi.$ The spectra $(\lambda_{n,j}^{m})$  of BVPs $D_{j}(P_{\beta},Q_{\beta},m) $ is independent of $\beta$ . 
\end{theorem}
\begin{proof}
Notice that in this case obviously $a_1+a_2<\pi$ and functions $K^m(x)$ and $G^m(x)$ have the form  (\ref{700}) and (\ref{770}) or (\ref{788}) and (\ref{789}) respectively, depending on the condition whether the second delay is less or not of $\frac{\pi}{2}$. Taking into account that  $p_1(x)=q_1(x)=0$, as well as that function  $q_2^\beta(x)$ is vanishing for $x\in(\pi-a_1,\pi),$ we obtain that 
\begin{equation}\label{388}
\begin{split}
K^{m}(x)=0, \;\; x\in (0, \pi),
\end{split}
\end{equation}
while in functions $G^m(x),$ besides functions $p_2(x)$ and $q_2(x),$ only integrals $\alpha_{p_{2}p_{2}}^{1}(x)$ and $\alpha_{p_{2}q_{2}}^{12}(x)$  remain.  
Since
\[p_2(x+\frac{a_1}{2})=0, \;x\in(0, \frac{a_1+a_2}{2})\cup(\pi-\frac{3a_1}{2},\frac{3a_1}{2}+\frac{a_2}{2})\] and 
\[q_2(x+\frac{a_2}{2})=0, \;x\; \in(0, a_1)\cup(\pi-a_1-\frac{a_2}{2},\pi-\frac{a_2}{2})\]
we obtain 
\[\alpha_{p_{2}p_{2}}^{1}(x)=0, \;\;x\in(\pi-a_1-\frac{a_2}{2},\pi-a_1)\] and
\[\alpha_{p_{2}q_{2}}^{12}(x)=0, \;x\in(\pi-\frac{3a_1}{2},\pi-\cfrac{a_1+a_2}{2}).\] 
Indeed, for $p_2(x)$  subintegral function's argument in integral $\alpha_{p_{2}p_{2}}^{1}(x)$ on the interval $(\pi-a_1-\frac{a_2}{2},\pi-a_1)$ is valid 
\[t-x<\pi-\pi+a_1+\frac{a_2}{2}=a_1+\frac{a_2}{2},\]
while $q_2(x)$  function's argument in  $\alpha_{p_{2}q_{2}}^{12}(x)$ on the interval $(\pi-\frac{3a_1}{2},\pi-\frac{a_1+a_2}{2})$ satisfies condition
\[t-x-\frac{a_1-a_2}{2}<\pi-\pi+\frac{3a_1}{2}-\frac{a_1-a_2}{2}=a_1+\frac{a_2}{2}.\]
Then we  have
\begin{equation}\label{389}
\begin{split}
G^{m}(x)&=0, \;\; x\in (\frac{a_1}{2}, a_1)\cup (\pi-\frac{3a_1}{2},\frac{3a_1}{2}+\frac{a_2}{2}),\\
G^{m}(x)&=(-1)^{m}\bigg(p_{2}(x+\frac{a_1}{2})-\alpha_{p_{2}q_{2}}^{12}(x)\bigg) \mathbbm{1}_{(\frac{a_1+a_2}{2},\pi-\frac{3a_1}{2})} (x)\\
+&\bigg (q_{2}(x+\frac{a_2}{2})- \alpha_{p_{2}p_{2}}^{1}(x)\bigg) \mathbbm{1}_{(a_{1},\pi-a_{1}-\frac{a_2}{2})} (x),\;\;  x\in (a_1 ,\pi-\frac{3a_1}{2}), \\
G^{m}(x)&=(-1)^{m}p_{2}(x+\frac{a_{1}}{2}),\;\;\;x\in (\frac{3a_1}{2}+\frac{a_2}{2},\pi-\frac{a_1}{2}).
\end{split}
\end{equation}
Then from (\ref{14}), (\ref{388}) and (\ref{389}), we obtain that characteristic functions $\Delta_{1}^{m}(\lambda)$ for family of functions $D_\beta$ have the form
\begin{equation*}
\begin{split}
\Delta_{1}^{m}(\lambda)=&\sin \lambda \pi + \int\limits_{\frac{a_{1}}{2}}^{\pi-\frac{a_{1}}{2}} K^{m}(x) \cos \lambda (\pi-2x)dx -
\int\limits_{\frac{a_{1}}{2}}^{\pi-\frac{a_{1}}{2}} G^{m}(x) \sin \lambda (\pi-2x)dx\\
=&\sin \lambda \pi-(-1)^{m}\int\limits_{\frac{a_{1}+a_2}{2}}^{\pi-\frac{3a_{1}}{2}} \bigg(p_{2}^\beta(x+\frac{a_1}{2})-\alpha_{p_{2}^\beta q_{2}^\beta}^{12}(x)\bigg) \sin \lambda (\pi-2x)dx\\
&-\int\limits_{a_1}^{\pi-a_1-\frac{a_2}{2}} \bigg(q_{2}^\beta(x+\frac{a_2}{2})\-\alpha_{p_{2}^\beta p_{2}^\beta}^{1}(x)\bigg) \sin \lambda (\pi-2x)dx \\
&-(-1)^{m}\int\limits_{\frac{3a_{1}}{2}+\frac{a_2}{2}}^{\pi-\frac{a_{1}}{2}} p_2^\beta(x+\frac{a_1}{2}) \sin \lambda (\pi-2x)dx\\
=&\sin \lambda \pi-(-1)^{m}\int\limits_{a_1+\frac{a_2}{2}}^{\pi-a_1} \bigg(p_{2}^\beta(x)-\alpha_{p_{2}^\beta q_{2}^\beta}^{12}(x-\frac{a_1}{2})\bigg) \sin \lambda (\pi-2x+a_1)dx
\end{split}
\end{equation*}
\[\begin{split}
&-\int\limits_{a_1+\frac{a_2}{2}}^{\pi-a_1} \bigg(q_{2}^\beta(x)-\alpha_{p_{2}^\beta p_{2}^\beta}^{1}(x-\frac{a_2}{2})\bigg) \sin \lambda (\pi-2x+a_2)dx\\
&-(-1)^{m}\int\limits_{2a_1+\frac{a_2}{2}}^{\pi} p_2^\beta(x) \sin \lambda (\pi-2x+a_1)dx.
\end{split}\]
Let us show that 
\begin{equation*}
 U(x)=p_{2}^\beta(x)-\alpha_{p_{2}^\beta q_{2}^\beta}^{12}(x-\frac{a_1}{2})=0,\;\; x\in (a_1+\frac{a_{2}}{2},\pi-a_1).   
\end{equation*}
Using the definition and properties of the integral operator $M(f)$ from (\ref{400}) and (\ref{301}), as well as  the form of functions $p_1^\beta(x)$ and $q_2^\beta (x)$ from (\ref{652}) and (\ref{654}), we obtain
\[\begin{split}
U(x)=&p_{2}^\beta(x)-\int\limits_{x+\frac{a_{2}}{2}}^{\pi} p_{2}^\beta(t)q_{2}^\beta(t-x+\frac{a_{2}}{2})dt\\
=&p_{2}^\beta(x)-\int\limits_{a_{2}}^{\pi-x+\frac{a_{2}}{2}} p_{2}^\beta (s+x-\frac{a_{2}}{2})q_{2}^\beta(s)ds\\=&p_{2}^\beta(x)-\int\limits_{a_1+\frac{a_2}{2}}^{\pi-x+\frac{a_{2}}{2}} q_{2}^\beta(s)p_{2}^\beta (s+x-\frac{a_{2}}{2})ds. 
\end{split}\]
Since 
\[\pi-x+\frac{a_2}{2}<\pi-a_1\;\wedge\;s+x-\frac{a_2}{2}>2a_1+\frac{a_2}{2}\]
we obtain 
\[\begin{split}
U(x)=&\beta e_{1}(x)-\int\limits_{a_{1}+\frac{a_2}{2}}^{\pi-x+\frac{a_{2}}{2}} \beta e_{1}(s)h(s+x-\frac{a_{2}}{2})ds\\
=&\beta e_{1}(x)-\beta M(e_{1}(x))=\beta e_{1}(x)-\beta e_{1}(x)=0.
\end{split}\]
\\
In the same way one can show that  
\begin{equation*}\
V(x)=q_{2}^\beta(x)-\alpha_{p_{2}^\beta p_{2}^\beta}^{1}(x-\frac{a_2}{2})=0,\;  x\in ( a_1+\frac{a_2}{2},\pi-a_1).    
\end{equation*}
Indeed,  we have
\[\begin{split}
V(x)&=q_{2}^\beta(x)-\int\limits_{x-\frac{a_2}{2}+a_1}^{\pi} p_{2}^\beta(t)p_{2}^\beta(t-x+\frac{a_2}{2})dt\\
&=q_{2}^\beta(x)-\int\limits_{a_{1}}^{\pi-x+\frac{a_2}{2}} p_{2}^\beta(s+x-\frac{a_2}{2})p_{2}^\beta(s)ds\\
&=q_{2}^\beta(x)-\int\limits_{a_{1}+\frac{a_2}{2}}^{\pi-x+\frac{a_2}{2}} p_{2}^\beta(s)p_{2}^\beta(s+x-\frac{a_2}{2})ds\\
&=\beta e_{1}(x)-\int\limits_{a_{1}+\frac{a_2}{2}}^{\pi-x+\frac{a_{2}}{2}} \beta e_{1}(s)h(s+x-\frac{a_{2}}{2})ds\\
&=\beta e_{1}(x)-\beta M(e_{1}(x))=\beta e_{1}(x)-\beta e_{1}(x)=0.  
\end{split}\]
Then we obtain
\[\begin{split}
\Delta_{1}^{m}(\lambda)=\sin \lambda \pi-(-1)^{m}\int\limits_{2a_1+\frac{a_2}{2}}^{\pi} p_2^\beta(x) \sin \lambda (\pi-2x+a_1)dx \\
=\sin \lambda \pi-(-1)^{m}\int\limits_{2a_1+\frac{a_2}{2}}^{\pi} h(x) \sin \lambda (\pi-2x+a_1)dx. 
\end{split}\]
In the same way from  (\ref{15}) we obtain 
\begin{equation*}
\Delta_{2}^{m}(\lambda)=
-\cos \lambda \pi  + (-1)^{m}\int\limits_{2a_1+\frac{a_{2}}{2}}^{\pi} h(x) \cos \lambda (\pi-2x+a_1)dx,
\end{equation*}
i.e. characteristic functions are independent of $\beta.$ Theorem is proved.
\end{proof}
{\bf Remark 3.1.} If we consider the case when the first delay is greater than the second one, then both delays are less then $\frac{2\pi}{5}$ and Theorem of uniqueness does not hold on the triangle which  completes the set $R$  up to the rectangle (Picture 3.1). But if  we do not limit the first delay to the interval $[\frac{\pi}{3},\frac{2\pi}{5})$ and consider the set
\[S=\{(a_1,a_2):\;\frac{\pi}{3}\leq a_2<a_1<\pi\}\] 
it is clear that the theorem of uniqueness holds on the subset 
\[S_1=\{(a_1,a_2)\in S: \;\frac{2\pi}{5} \leq a_2<a_1\}.\]  
On the subset
\[S_2=S\setminus S_1=\{(a_1,a_2) \in S: \frac{\pi}{3}\leq a_1<\pi \wedge \frac{\pi}{3}<a_2<\frac{2\pi}{5} \wedge a_2<a_1\}\]  
the theorem of uniqueness does not hold which follows from the results for Dirac operator with one delay. Indeed, if we take $P(x)=0$ in (\ref{a01}), then we obtain two BVPs with one delay $a_2.$  It is known that inverse problem's solution is not unique in the case $a_2<\frac{2\pi}{5}$  (see \cite{dju}) and then we conclude that the theorem of uniqueness does not hold on the subset $S_2.$
Taking into account this result, as well as the result from the Theorem 3.2, we conclude that position of $(-1)^m$ in  the BVPs setting is very important and significantly affects the size of the set on which the theorem of uniqueness is valid.

\medskip
\noindent
{\bf Biljana Vojvodi\'{c}}\\
Faculty of Mechanical Engineering,\\
University of Banja Luka, \\
Banja Luka, Bosnia and Herzegovina. \\
E-mail: biljana.vojvodic@mf.unibl.org

\medskip
\noindent
{\bf Nebojša Djurić}\\
Faculty of Electrical Engineering,\\
University of Banja Luka, \\
Banja Luka, Bosnia and Herzegovina. \\
E-mail: nebojsa.djuric@etf.unibl.org

\medskip
\noindent
{\bf Vladimir Vladi\v{c}i\'{c}}\\
Faculty of Philosophy,\\
University of East Sarajevo, \\
East Sarajevo, Bosnia and Herzegovina. \\
E-mail: vladimir.vladicic@ff.ues.rs.ba

\end{document}